\documentclass{alggeom}

\usepackage{amsmath}
\usepackage{amssymb}
\usepackage{color}
\usepackage{amsthm}
\usepackage{tikz}
\usepackage{tikz-cd}
\usetikzlibrary{positioning,calc,cd}
\usepackage{hyperref}
\usepackage{tensor}
\usepackage{relsize}
\usepackage{mathtools}
\usepackage{quiver}

\newtheorem{theorem}{Theorem}[section]
\newtheorem{proposition}[theorem]{Proposition}
\newtheorem{lemma}[theorem]{Lemma}
\newtheorem{corollary}[theorem]{Corollary}

\newtheoremstyle{defis}%
    {3pt}
    {3pt}
    {}
    {}
    {\scshape}
    {.}
    {.5em}
    {}
\theoremstyle{defis}
\newtheorem{definition}[theorem]{Definition}
\newtheorem{example}[theorem]{Example}
\newtheorem{remark}[theorem]{Remark}

\definecolor{Noir}{rgb}{0,0,0} 
\definecolor{Blanc}{rgb}{1,1,1} 
\definecolor{Gray}{rgb}{0.5,0.5,0.5} 
\definecolor{Rouge}{rgb}{0.8,0.1,0.1} 
\definecolor{DBleu}{RGB}{51,51,178} 
\definecolor{LBleu}{rgb}{0.85,0.85,1} 
\definecolor{Orange}{RGB}{255,140,0} 
 
\newcommand{\bcent}{\begin{center}} 
\newcommand{\ecent}{\end{center}} 
 
\newcommand{\benum}{\begin{enumerate}} 
\newcommand{\eenum}{\end{enumerate}} 
\newcommand{\bitem}{\begin{itemize}} 
\newcommand{\eitem}{\end{itemize}} 
 
\newcommand{\btab}{\begin{tabular}} 
\newcommand{\etab}{\end{tabular}} 
\newcommand{\beqn}{\begin{eqnarray}} 
\newcommand{\eeqn}{\end{eqnarray}} 
 
\newcommand{\bmath}{\begin{math}} 
\newcommand{\emath}{\end{math}} 
 
\newcommand{\noin}{\noindent} 

\providecommand{\F}[1]{\mathbb{#1}}

\newcommand{\ZZ}{\mathbb{Z}}

\newcommand{\QQ}{\F{Q}} 
\newcommand{\RR}{\F{R}} 
\newcommand{\CC}{\F{C}}

\newcommand{\PP}{\F P}

\newcommand{\calC}{\mathcal{C}}
\newcommand{\calF}{\mathcal{F}}

\newcommand{\calS}{\mathcal{S}}

\newcommand{\calT}{\mathcal{T}}
\newcommand{\calG}{\mathcal{G}}
\newcommand{\calX}{\mathcal{X}}
\newcommand{\calZ}{\mathcal{Z}}

\renewcommand{\epsilon}{\varepsilon}

\newtheorem{manualtheoreminner}{Theorem}
\newenvironment{manualtheorem}[1]{%
  \renewcommand\themanualtheoreminner{#1}%
  \manualtheoreminner
}{\endmanualtheoreminner}

\newcommand\varleq{\mathbin{\vcenter{\hbox{%
  \oalign{\hfil$\scriptstyle<$\hfil\cr 
          \noalign{\kern-.3ex} 
          $\scriptscriptstyle({-})$\cr}%
}}}} 
 
\renewcommand\subsetneq{\mathbin{\vcenter{\hbox{%
  \oalign{\hfil$\scriptstyle\subset$\hfil\cr 
          \noalign{\kern-.3ex} 
          $\scriptscriptstyle({-})$\cr}%
}}}} 

\author{Evan Sundbo}
\email{evan.sundbo@mail.utoronto.ca}
\address{Department of Mathematics, University of Toronto\\ Toronto ON, Canada}

\title[Broken Toric Varieties and Cell-Compatible Sheaves]{Broken Toric Varieties and Cell-Compatible Sheaves}
\date{\today}
\subjclass[2010]{14F06, 14M25}

\begin{document}

\begin{abstract}
We study the cohomology of broken toric varieties via the derived push-forward of the constant sheaf to a complex of polytopes, proving a Deligne-type decomposition theorem, degeneration of the associated Leray-Serre spectral sequence, and showing that the Leray filtration on their cohomology is equal to twice the weight filtration. Furthermore, we give an explicit formula for the Betti numbers of some broken toric varieties whose associated complex of polytopes is the $n$-skeleton of a higher dimensional polytope, encompassing some important examples.

\end{abstract}
\maketitle

\tableofcontents

\section{Introduction}\label{1}

Broken toric varieties are a class of varieties, essentially combinatorial in nature, which appear in areas such as mirror symmetry, hypertoric geometry, toric topology, and the study of Hitchin systems. To be precise, a \emph{broken toric variety} (or \emph{stable toric variety}) is a union of (smooth projective) toric varieties glued to each other pairwise along $T$-invariant toric subvarieties. They are combinatorial in the sense that their geometry is entirely captured by a complex of polytopes, which we accordingly call their \emph{polytope complex}, in the same way that the information of a toric variety is studied by looking at its associated polytope or fan.

There is a map from a broken toric variety to its polytope complex, viewed alternatively as the quotient map by the compact part of the torus action, or as the moment map for that action. By pushing forward the constant sheaf on the broken toric variety along this map, we obtain a complex of sheaves on the polytope complex with nice combinatorial properties (its cohomology objects are \emph{cell-compatible}), which make them amenable to computation.

Broken toric varieties can be thought of degenerations of toric varieties (a nice overview of this point of view can be found in \cite{A15}) and have previously attracted a bit of attention. After their introduction in \cite{A02}, the structure of their moduli space is investigated in \cite{AB06} and \cite{AM16}.  The main component of the moduli space is compactified in \cite{O12} and tropicalized in \cite{MW?}. 

The organization of the article is as follows: In Section \ref{2}, we state the relevant definitions, give some examples, and prove that the higher derived pushforwards of the constant sheaf on a broken toric variety along its quotient map $f$ to its polytope complex are all cell-compatible subsheaves of a constant sheaf. Our main result appears in Section \ref{3}, which is that the derived pushforward $Rf_*\underline{\QQ}_X$ is a formal complex and so the Leray spectral sequence degenerates. Examples of non-smooth varieties for which this classical version of the Decomposition Theorem holds are fairly rare.

\begin{manualtheorem}{3.4}[]
For $X$ an $n$-dimensional broken toric variety and $f$ the map from $X$ to its polytope complex $P_\bullet$, there is an isomorphism in $\mathcal{D}^b_c(P_\bullet)$
$$Rf_*\underline{\QQ}_X \cong \bigoplus_{i=0}^{2n}R^if_*\underline{\QQ}_X[-i].$$
This implies the degeneration at the $E_2$ page of the Leray spectral sequence $E^{pq}_2 = H^p(P_\bullet, R^qf_*\underline{\QQ}_X)$ associated to $f$.
\end{manualtheorem}

In Section \ref{4} we show that the Leray filtration associated to $f$ on the cohomology of a broken toric variety is equal to twice the weight filtration.

\begin{manualtheorem}{4.2}[]
For $X$ a broken toric variety, $$W_{2k}H^i(X,\underline{\QQ}_X) = W_{2k+1}H^i(X,\underline{\QQ}_X) = L_{k}H^i(X,\underline{\QQ}_X).$$
\end{manualtheorem}

The final two sections are together devoted to the study of skeletal polytope complexes, i.e. those polytope complexes which are the $n$-skeleton of a higher dimensional polytope. Section \ref{5} studies cell-compatible sheaves on this higher dimensional polytope, culminating in a new way to derive the Betti numbers of a smooth projective toric variety (Corollary \ref{toricbetti}).  Finally, in Section \ref{6} we explain how this relates to the Betti numbers of broken toric varieties with skeletal polytope complexes. In particular this includes many polytope complexes which are polytopal decompositions of an $n$-sphere.

As alluded to above, broken toric varieties make appearances in many areas.  Fine compactified Jacobians of nodal curves with all rational components are broken toric, as studied in \cite{OS79}. This immediately points to a relationship with the moduli space of Higgs bundles or Hitchin system (see \cite{H87, S92, R18, J21} and the many references within), the fibres of which above nodal spectral curves are precisely their fine compactified Jacobians \cite{BNR89}.  Relatedly, the special fibre of a hypertoric Hitchin system \cite{HP06, DMS19, GM20} is also broken toric of the same form. In particular, they are broken toric with polytope complexes which are polytopal decompositions of an $n$-torus, and whose $n-1$ cells can be interpreted as a periodic arrangement of hyperplanes in $\RR^n$.

In a different direction, broken toric varieties whose polytope complex is a polytopal decomposition of an $n$-sphere are of interest to the Gross-Siebert program \cite{GW00, GS06, KX16} in mirror symmetry.  Roughly speaking, one wants to study mirror symmetry for families of Calabi-Yaus $\calX \to S$ by looking at certain maximally singular fibres, the so-called large complex structure limits. These are in general broken toric varieties (with additional data), and it is conjectured that their polytope complexes are polytopal decompositions of an $n$-sphere.  This is trivially true for elliptic curves and proven for K3 surfaces in \cite{GW00}. 

One can find broken toric varieties when studying complexity $k$ $T$-varieties \cite{AIPSV12}, where the fibres of the natural quotient map, after resolving indeterminacies, are broken toric \cite{I?}.  They also appear in the field of toric topology \cite{BR08, BP15, BBC20}.  For example, the characteristic function of a polytope as defined in \cite{DJ91} describes the stalks of our cell-compatible sheaves (see the proof of Proposition \ref{cc}). In this area one also studies polyhedral products \cite{BBCG09}, topological spaces determined by a simplicial complex and a family of pairs of pointed topological spaces. Polyhedral products are a generalization of the moment angle complexes of \cite{DJ91, BP00} and their specializations find applications in many areas (see in particular the table in Chapter $1$ of \cite{BBC20}). It turns out that broken toric varieties also fit into this formalism, which we explain in Appendix \ref{A}. 
\\

\noin\emph{Acknowledgements.} The author would like to thank Tony Bahri, Yael Karshon, and Michael McBreen for interesting discussions regarding this work.  Moreover, the input of Michael Groechenig was absolutely indispensable, and this paper would not exist without him. The input of the anonymous reviewer has also proven most useful. This work was supported by the FAST Doctoral Scholarship at the University of Toronto.\\

\section{Definitions and Derived Pushforwards}\label{2}

For us, a toric variety will always mean a smooth projective toric variety over $\CC$, and $T$ will be the maximal compact subgroup of the algebraic torus acting on it. A polytope is the intersection of a finite number of half-spaces in some $\RR^n$, and is simple if exactly $n$ codimension $1$ faces meet at each vertex.  It is a theorem of Delzant \cite{D88} that there is a bijective correspondence between symplectic toric varieties and simple polytopes for which  the edges at each vertex form a $\ZZ$-basis of $\ZZ^n$. We further recall that any smooth projective toric variety is symplectic since one can pull back the Fubini-Study form from the ambient projective space. The action of $T$ provides an effective Hamiltonian group action, so a symplectic toric variety is equipped with a symplectic moment map $f:X\to \RR^n$, the image of which is a polytope in $\RR^n$ \cite{GS82}. Indeed, this is the polytope which appears in Delzant's correspondence. Since we are interested in the topological properties of such objects, we can ignore the Delzant condition on the vertices.

\begin{definition} A \emph{broken toric variety} is a union of smooth projective toric varieties glued to each other pairwise along $T$-invariant toric subvarieties. The \emph{polytope complex} $P_\bullet(X)$ of a broken toric variety $X$ is the union of the polytopes of the components of $X$, glued to each other so that if two components of X meet along a toric variety $X'$, then the polytopes of those components intersect as the polytope of $X'$.
\end{definition}

There is a natural map $f:X\to P_\bullet(X)$ given by gluing together the moment maps of the components of $X$ to their polytopes.

\begin{remark}
To any complex of polytopes $P_\bullet$ there is a not only one, but rather a family of broken toric varieties which have $P_\bullet$ as their associated polytope complex. Given $P_\bullet$, a broken toric variety $X$ with components $X_i$ is fixed by choosing gluing isomorphisms $\alpha_{ij}$ from a toric subvariety in $X_i$ to an isomorphic subvariety in $X_j$ (satisfying the cocycle condition $\alpha_{ij}\alpha_{jk} = \alpha_{ik}$). This is equivalent to the choice of an element of $H^1(P_\bullet,T^n_\CC)$, that is, to the choice of a $T^n$-torsor on $P_\bullet$. We further note that this is not necessarily a bijection of $H^1(P_\bullet,T^n_\CC)$ with the space of isomorphism classes of broken toric varieties over $P_\bullet$; for example, if $P_\bullet$ is the graph in Example \ref{ex1}, $H^1(P_\bullet,T^n_\CC)$ is one-dimensional but all broken toric varieties over it are isomorphic.
\end{remark}

\begin{lemma}\label{quotient}
Any $n$-dimensional broken toric variety $X$ is a quotient of the total space of a $T^{n}$-torsor on $P_\bullet$.
\end{lemma}

\begin{proof}
By Lemma 1.4 of \cite{DJ91}, any smooth projective toric variety $X$ with polytope $P$ can be written as a quotient of $P\times T^n$.  The proof in fact applies verbatim to broken toric varieties whose polytope complexes $P_\bullet$ are contractible.  The idea is to inductively blow up the singular strata of $X$ and obtain a $T^n$-manifold $\hat{X}$ whose orbit space is identifed with $P_\bullet$, as well as a natural "collapse" map $\hat{X}\to X$. Since $P_\bullet$ contractible, $\hat{X}$ is diffeomorphic to $P_\bullet\times T^n$.

For the general case, let $\calT$ be the total space of $T^n$-torsor defining how the toric components of $X$ are glued together. Locally around any $n-1$ cell of $P_\bullet$, $\calT$ is the trivial torsor, so we can define a map from this local neighbourhood to a neighbourhood of $P_\bullet$ by using the above construction.  The result is a quotient map from $\calT$ to a broken toric variety over $P\bullet$ whose components are glued as prescribed by the $T^n$-torsor defining $X$ over $P_\bullet$, which is necessarily isomorphic to $X$.
\end{proof}

\begin{remark}
The proof of Lemma 1.4 of \cite{DJ91} depends on the existence of a smooth $T^n$ action on the toric variety, and as such does not hold for singular toric varieties. This is the main reason for our insistence that the toric components of a broken toric variety be smooth.  Our upcoming decomposition result (Theorem \ref{ff}) depends on Lemma \ref{quotient}, and so does not hold for the case of singular components. It's even worse, in fact: an explicit example is given in \cite{McC89} of two singular toric varieties with the same polytope but different rational Betti numbers.   It would be interesting to investigate a version of the decomposition theorem for broken toric varieties with singular components using intersection cohomology. 
\end{remark}

\begin{definition}
Let $X$ be a topological space with a CW complex structure $\mathcal{E} = \{\text{Sk}_i\}_{i=0}^{\text{dim}X}$ and denote by $\mathring{\text{Sk}}_i$ the complement of $\text{Sk}_{i-1}$ in $\text{Sk}_i$. A constructible sheaf $\mathcal{F}$ on $X$ is \emph{cell-compatible} with respect to $\mathcal{E}$ if for all $\alpha\in\mathring{\text{Sk}}_i$, $\calF|_{\alpha}$ is constant on $\alpha$ and if for all $x\in\alpha\in\mathring{\text{Sk}}_i$, the stalk $\calF_x$ lies in the intersection $\bigcap_{y\in\beta\in\mathring{\text{Sk}}_{i+1}}\calF_y$, where the intersection runs over all $\beta\in\mathring{\text{Sk}}_{i+1}$ for which $\alpha\in\partial\beta$.
\end{definition}

For the technically minded reader, the intersection in this definition can be understood in the following way. Let $x\in\alpha\in\mathring{\text{Sk}}_i$, $y_1\in\beta_1\in\mathring{\text{Sk}}_{i+1}$, and $y_2\in\beta_2\in\mathring{\text{Sk}}_{i+1}$ with $\alpha\in\partial\beta_1\cap \partial\beta_2$. Let $\gamma_1$ be a morphism from $x$ to $y_1$ in the category of exit paths \cite{M?, T09} of $\calF$ (and similarly for $\gamma_2$).  We say that $\calF_x\subseteq \calF_{y_1}\cap\calF_{y_2}$ if the constructible sheaves $\gamma_1^{-1}\calF$ and $\gamma_2^{-1}\calF$ have constant subsheaves $\calG_1$ and $\calG_2$ respectively such that $(\calG_1)_x = \calF_x = (\calG_2)_x$.

The following proposition shows how our two previous definitions are related; indeed, our main examples of cell-compatible sheaves come from broken toric varieties.

\begin{proposition}\label{cc} Let $X$ be an $n$-dimensional broken toric variety and $\calT\xrightarrow[]{\pi} P_\bullet$ be the the total space of the $T^n$-torsor on $P_\bullet$ describing the gluing data of its toric components. The higher derived pushforwards $R^if_*\underline{\QQ}_X$ of $\underline{\QQ}_X$ along $f$ are subsheaves of $R^i\pi_*\calT$ which are cell-compatible with respect to the CW complex structure on $P_\bullet$.
\end{proposition}

\begin{proof}

For cell-compatibility, Let us start with a toric variety $X$ with polytope $P$. Choose a basis $\{e_1,\dots,e_m\}$ of the ambient real vector space in which $P$ lies.  Then for all $k$ and $\alpha\in\mathring{\text{Sk}}_k(P)$, choose a basis $\{d^\alpha_1,\dots,d^\alpha_k\}$ of the sub-vector space defined by translating $\alpha$ to intersect $0$, with each $d^\alpha_i$ written in terms of the chosen basis vectors $e_i$. The fibre $f^{-1}(x)$ over a point $x\in\alpha\in\mathring{\text{Sk}}_k(P)$ can be described as $\langle d^\alpha_1,\dots,d^\alpha_k \rangle/\Gamma$, where $\Gamma$ is the standard lattice.

Now proper base change allows us to describe the stalks of $R^if_*\underline{\QQ}_X$ as \begin{equation}\label{stalkdescr}
(R^if_*\underline{\QQ}_X)_x = H^i(f^{-1}(x),\underline{\QQ}_{f^{-1}(x)}) = \bigwedge\nolimits^i \langle d^\alpha_1,\dots,d^\alpha_k \rangle
\end{equation}
from which cell-compatibility follows. For the broken case one only has to note that identifying the stalks of two cell-compatible sheaves along a cell, corresponding to the gluing of toric components, preserves cell-compatibility.

To show that $R^if_*\underline{\QQ}_X$ is a subsheaf of $R^i\pi_*\calT$, consider the commutative diagram
\[
\begin{tikzcd}
   \calT \arrow{rr}{r} \arrow[swap]{dr}{\pi} & & X \arrow{dl}{f} \\[10pt]
    & P_\bullet
\end{tikzcd}
\]
where $\pi$ is the projection onto $P_\bullet$ and $r$ is the quotient map whose existence is assured by Proposition \ref{quotient}.

Denote the sheaf complex $Rr_*\underline{\QQ}_{\calT}$ by $\calG^\bullet$.  There is a natural map of (complexes of) sheaves on $X$ given by $\alpha:\underline{\QQ}_X\to\calG^\bullet$, obtained by adjunction from the identity map $r^{-1}\underline{\QQ}_X\to \underline{\QQ}_{\calT}$. This map descends to a map of complexes on $P_\bullet$, $R\alpha: Rf_*\underline{\QQ}_X\to Rf_*\calG^\bullet$ and further to a map on the cohomology sheaves $$R^i\alpha:R^if_*\underline{\QQ}_X\to R^if_*\calG^\bullet = R^i\pi_*\calT,$$ which we will show is injective by checking injectivity on the stalks.

For $y\in\mathring{\text{Sk}}_j$, $(R^i\alpha)_y:\bigwedge\nolimits^i\QQ^j\to\bigwedge\nolimits^i\QQ^{n}$, and we note that $r|_{\pi^{-1}(y)}:T^n\twoheadrightarrow T^j$, which implies that $(R^i\alpha)_y$ is injective. Since this holds for all $0\leq j\leq n$, $R^i\alpha$ is injective.
\end{proof}

Now we have the pleasure of showing some examples. 

\begin{example}\label{ex1}
First consider the case of a necklace of three copies of the projective line, with $0$ in one glued to $\infty$ in the next. The polytope complex of this broken toric variety is a triangulation of $S^1$ with three $0$-cells which we can embed in $\RR^2$ with the standard basis $\{e_1,e_2\}$. In the following picture we have labelled some of the stalks of $R^1f_*\underline{\QQ}_X$.

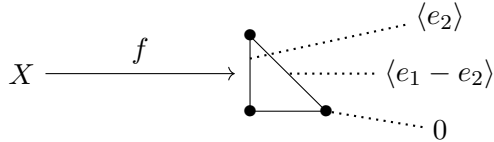
\begin{center}
\begin{tikzpicture}
\node (a) at (0,0){$X$};
\draw (3,-0.5) -- (4,-0.5) -- (3,0.5) -- cycle;
\node (b) at (3,-0.5){$\bullet$};
\node (c) at (4,-0.5){$\bullet$};
\node (e) at (3,0.5){$\bullet$};
\draw[->] (a) -- (2.8,0) node[pos=0.5,above
]{$f$};

\node (f) at (5.5,0.75){$\langle e_2 \rangle$};
\node (g) at (5.5,0){$\langle e_1-e_2 \rangle$};
\node (h) at (5.5,-0.75){$0$};

\draw[thick, -, dotted] (f) -- (3,0.2);
\draw[thick, -, dotted] (g) -- (3.5,0);
\draw[thick, -, dotted] (h) -- (4,-0.5);

\end{tikzpicture}
\end{center}
\end{example}

\begin{example}\label{ex2}
For a higher-dimensional case, consider two copies of $\PP^2$ glued to each other along a copy of $\PP^1$.  In this case the polytope complex of $X$ is two triangles glued to each other along one edge, which we can again embed in $\RR^2$. As above, we have labelled some of the stalks of $R^1f_*\underline{\QQ}_X$.

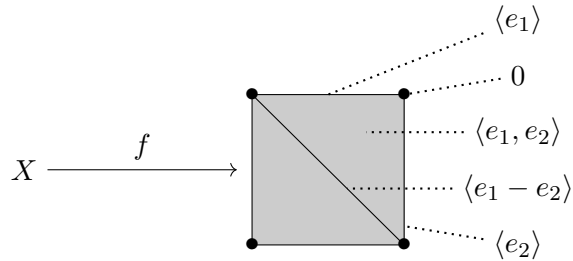
\begin{center}
\begin{tikzpicture}
\node (a) at (0,0){$X$};
\fill[gray!40!white] (3,-1) rectangle (5,1);
\draw (3,-1) -- (5,-1) -- (5,1) -- (3,1) -- cycle;
\draw (3,1) -- (5,-1);
\node (b) at (3,-1){$\bullet$};
\node (c) at (5,-1){$\bullet$};
\node (d) at (5,1){$\bullet$};
\node (e) at (3,1){$\bullet$};
\draw[->] (a) -- (2.8,0) node[pos=0.5,above
]{$f$};

\node (f) at (6.5,2){$\langle e_1 \rangle$};
\node (g) at (6.5,1.25){0};
\node (h) at (6.5,0.5){$\langle e_1,e_2 \rangle$};
\node (i) at (6.5,-0.25){$\langle e_1 - e_2 \rangle$};
\node (j) at (6.5,-1){$\langle e_2 \rangle$};

\draw[thick, -, dotted] (f) -- (4,1);
\draw[thick, -, dotted] (g) -- (5,1);
\draw[thick, -, dotted] (h) -- (4.5,0.5);
\draw[thick, -, dotted] (i) -- (4.25,-0.25);
\draw[thick, -, dotted] (j) -- (5,-0.75);
\end{tikzpicture}
\end{center}

\end{example}

\section{Decomposition and Vanishing Results}\label{3}

In this section we establish some of our key tools, all of which apply to any broken toric variety. This first lemma is not surprising and finds uses throughout the article.

\begin{lemma}\label{openclosed} For $P$ be a CW complex let us denote by $\iota_k$ the inclusion of $\mathring{\text{Sk}}_k$ into $\text{Sk}_{k}$ and by $\kappa_k$ the inclusion of $\text{Sk}_k$ into ${\text{Sk}_{k+1}}$. For $\calF$ a cell-compatible sheaf on $P$ and $0\leq k \leq n$, there is a short exact sequence
\begin{equation}\label{restrictiontoskeletongeneral}
0\to {\iota_k}_!{\iota_k}^{-1}\calF\to {\kappa_k}_!\kappa_k^{-1}\calF \to (\kappa_{k-1})_!(\kappa_{k-1})^{-1}\calF\to 0.
\end{equation}
Moreover, when $\calF = R^if_*\underline{\QQ}_X$ for a broken toric variety $X$, one finds ${\iota_k}_!{\iota_k}^{-1} \calF\cong {\iota_k}_!\bigwedge\nolimits^i\underline{\QQ}_{\mathring{\text{Sk}}_{k}(P_\bullet)}^{k}$.
\end{lemma}

\begin{proof}
Recall that for a sheaf $\calF$ on a topological space $X$ with $i:Z\hookrightarrow X$ a closed subspace and $j:U\hookrightarrow X$ its complementary open subspace, there is an exact sequence of sheaves
$$0\to j_!j^{-1}\calF \to \calF \to i_*i^{-1}\calF\to 0.$$
The sequence in the statement of the lemma nothing more than this sequence for the complementary open and closed subspaces, in this case to the subspaces $\text{Sk}_{k-1}$ and $\mathring{\text{Sk}}_k$ of ${\text{Sk}_k}$.

To show that  ${\iota_k}_!{\iota_k}^{-1} \calF|_{\text{Sk}_{k}}\cong {\iota_k}_!\bigwedge\nolimits^i\underline{\QQ}_{\mathring{\text{Sk}}_{k}(P_\bullet)}^{k}$, note that ${\iota_k}^{-1} \calF|_{\text{Sk}_{k}}$ is a sheaf on $\mathring{\text{Sk}}_k$ (a union of open $k$-cells) whose stalks are $\bigwedge\nolimits^i\QQ^k$ by proper base change. Hence it is isomorphic to $\bigwedge\nolimits^i\underline{\QQ}_{\mathring{\text{Sk}}_{k}(P_\bullet)}^{k}$.
\end{proof}

The following lemma describes a general vanishing result.

 \begin{lemma}\label{j<ivanishing}
For $X$ a broken toric variety, $H^j(P_\bullet,R^if_*\underline{\QQ}_X) = 0$ for all $j<i$.
\end{lemma}

\begin{proof}
The proof is by induction on $k$ using Lemma \ref{openclosed}. When $k<i$ the sheaf ${\kappa_k}_!{\kappa_k}^{-1}R^if_*\underline{\QQ}_X$ is zero, and for $k=i$ we find that ${\kappa_k}_!{\kappa_k}^{-1}R^if_*\underline{\QQ}_X
\cong {\iota_k}_!\underline{\QQ}_{\mathring{\text{Sk}}_k(P_\bullet)}$.
Thus
\begin{equation*}
H^j(P_\bullet, {\kappa_k}_!{\kappa_k}^{-1}R^if_*\underline{\QQ}_X) = \begin{cases}
\QQ^{|\mathring{\text{Sk}}_i(P_\bullet)|} , &j=i\\
0, &\text{otherwise}
\end{cases}
\end{equation*}
We now invoke the long exact sequence in cohomology associated to \eqref{restrictiontoskeletongeneral}:
\begin{equation}
\cdots\to H^j({\iota_k}_!\underline{\QQ}_{\mathring{\text{Sk}}_{k}(P_\bullet)}^{\binom{k}{i}})\to H^j({\kappa_k}_!{\kappa_k}^{-1}R^if_*\underline{\QQ}_X) \to H^j((\kappa_{k-1})_!(\kappa_{k-1})^{-1}R^if_*\underline{\QQ}_X) \to \cdots.
\end{equation}
Since the cohomology of $ {\iota_k}_!\underline{\QQ}_{\mathring{\text{Sk}}_{k}(P_\bullet)}$ vanishes in all degrees except $k$, and by inductive hypothesis the cohomology of $(\kappa_{k-1})_!(\kappa_{k-1})^{-1}R^if_*\underline{\QQ}_X$ vanishes in all degrees less than $i$, this sequence shows that $H^j(P_\bullet,{\kappa_k}_!{\kappa_k}^{-1}R^if_*\underline{\QQ}_X) =0$ for $j<i$.  The result follows by noticing that ${\kappa_n}_!{\kappa_n}^{-1}R^if_*\underline{\QQ}_X \cong R^if_*{\QQ}_X$.

\end{proof}

Before we move on to Theorem \ref{ff} we will describe a certain endomorphism which is needed for our proof.

\begin{lemma}\label{end}
Let $X$ be a broken toric variety with polytope complex $P_\bullet$ which is the quotient of a $T^n$-torsor on $P_\bullet$ whose non-torsion part is trivial. Then there exists a positive integer $N$ and an endomorphism $[N]$ of $Rf_*\underline{\QQ}_{X}$ whose induced endomorphism on $\text{Ext}^{\,j}(\calF,R^if_*\underline{\QQ}_X)$ for any $\calF\in D^b_c(X)$ acts as multiplication by $N^i$.  
\end{lemma}

\begin{proof}
If $X$ is a quotient of $T^n\times P_\bullet$ we can define $[N]:T^n\times P_\bullet \to T^n\times P_\bullet$ by $(x,y)\mapsto (x^N,y)$.  Then this action descends to an action on the $X$ which preserves the fibres of $f:X \to P_\bullet$ by definition.  

Otherwise the torsor $\calT$ is torsion, meaning that it corresponds to a cocycle in Čech cohomology $\{\alpha_{ij}\} = \alpha \in H^1(P_\bullet, T^n_\CC)$ for which $\alpha^m=1$ for some integer $m>1$. We can assume that $\{\alpha_{ij} \}$ actually forms a cocycle in $H^1(P_\bullet, T^n_\CC[m])$ by looking at the long exact sequence associated to $0\to T^n_\CC[m]\to T^n_\CC \xrightarrow[]{\times m} T^n_\CC \to 0$, namely
$$\dots\to H^1(T^n_\CC[m]) \to H^1(T^n_\CC[m]) \xrightarrow[]{\times m} H^1(T^n_\CC[m]) \to \dots$$
 Since $\alpha$ lies in the kernel of the multiplication by $m$ map, we can pull it back to a a cocycle in $H^1(P_\bullet, T^n_\CC[m])$ and so $\alpha_{ij}^m=1$. Now choose a positive integer $N$ which is congruent to $1$ modulo $m$ and define $[N]$ acting on $\calT$ as follows:  Locally on $T^n\times U_i$, let $[N]$ act as $(x,y)\mapsto (x^N,y)$. This action glues to form a global action since if $(x,y) = \alpha_{ij}(x',y')$, then $[N](x,y) = \alpha_{ij}[N](x',y')$ as required. This action then yields an action on $X$.

With $[N]$ so defined, let us describe how it descends to $Rf_*\underline{\QQ}_X$. Adjunction provides a map $\underline{\QQ}_{X} \to R[N]_*R[N]^*\underline{\QQ}_{X} = R[N]_*\underline{\QQ}_X$ which we can compose with $Rf_*$ to get a map $Rf_*\underline{\QQ}_X\to Rf_*R[N]_*\underline{\QQ}_X$. Consider then the diagram 
\[
\begin{tikzcd}
   X \arrow{rr}{[N]} \arrow[swap]{dr}{f} & & X \arrow{dl}{f} \\[10pt]
    & {P_\bullet}
\end{tikzcd}
\]
from which we know that $Rf_*\underline{\QQ}_X = Rf_*R[N]_*\underline{\QQ}_X$. Thus, we have an endomorphism of $Rf_*\underline{\QQ}_X$ which we also call $[N]$. Further, this induces endomorphisms $[N]$ of the sheaves $R^if_*\underline{\QQ}_X$. By proper base change, we know how $[N]$ acts on the stalks of $R^if_*\underline{\QQ}_X$.  In particular, recall that $(R^if_*\underline{\QQ}_X)_x = H^i(f^{-1}(x),\underline{\QQ}_X|_{f^{-1}(x)}) = H^i(T^{k})$ for $x\in\mathring{\text{Sk}}_k(P_\bullet)$. Since the action on the fibres of $f$ was induced directly from the map $[N]$ on $T^n\times {P_\bullet}$, $[N]$ acts as the $N$-th power map on $f^{-1}(x)$.  Hence, the action on $H^i(T^{k})$ is given by multiplication by $N^i$.  In particular, $[N]$ acts as multiplication by $N^i$ on $R^if_*\underline{\QQ}_X$ and similarly on $\text{Ext}^{\,j}(\calF,R^if_*\underline{\QQ}_X)$. 

\end{proof}

\begin{theorem}\label{ff}
For $X$ an $n$-dimensional broken toric variety and $f$ the map from $X$ to its polytope complex $P_\bullet$, there is an isomorphism in $\mathcal{D}^b_c(P_\bullet)$
$$Rf_*\underline{\QQ}_X \cong \bigoplus_{i=0}^{2n}R^if_*\underline{\QQ}_X[-i].$$
This implies the degeneration at the $E_2$ page of the Leray spectral sequence associated to $f$
$$E^{pq}_2 = H^p(P_\bullet, R^qf_*\underline{\QQ}_X).$$
\end{theorem}

\begin{proof} This proof consists of two main steps, the first being to show that the statement holds for a broken toric variety whose gluing data consists of a $T^n$-torsor whose non-torsion part is trivial, using the endomorphism provided by Lemma \ref{end}. The second part of the proof is to notice that each component of the space $H^1(P_\bullet,T^n_\CC)$ indexing different possible broken toric varieties over $P_\bullet$ contains such a torsor, and that the derived pushforward complex does not vary in components.

The strategy for the first part of the proof will be to show that \begin{equation}\label{homdecomp}
\text{Hom}(-,Rf_*\underline{\QQ}_X) = \bigoplus_{i=0}^{2n} \text{Hom}(-,R^if_*\underline{\QQ}_X[-i]),
\end{equation}
 which yields the desired decomposition by the Yoneda Lemma. 
 
Truncating the complex $Rf_*\underline{\QQ}_X$ yields distinguished triangles 
$$\tau_{\leq i}Rf_*\underline{\QQ}_{X} \to \tau_{\leq i+1}Rf_*\underline{\QQ}_{X}  \to \tau_{\geq i+1}\tau_{\leq i+1}Rf_*\underline{\QQ}_{X}\xrightarrow[]{+1}$$
for $i=0,\dots, n-1$.  Each of these yields a long exact sequence after applying the functor $\text{Hom}(\calF,-)$ for any $\calF\in D^b_c(X)$
$$\ldots\to\text{Ext}^j(\calF,\tau_{\leq i}Rf_*\underline{\QQ}_{X}) \to\text{Ext}^j(\calF, \tau_{\leq i+1}Rf_*\underline{\QQ}_{X} ) \to \text{Ext}^j(\calF, R^{i+1}f_*\underline{\QQ}_X[-(i+1)])\to\ldots$$

Now we claim that the connecting homomorphisms 
\begin{equation*}\text{Ext}^j(\calF, R^{i+1}f_*\underline{\QQ}_X[-(i+1)])  \xrightarrow[]{\delta_j}\text{Ext}^{j+1}(\calF,\tau_{\leq i}Rf_*\underline{\QQ}_{X})
\end{equation*}
are zero for all $i,j$. This will be sufficient to show \eqref{homdecomp} because it implies in particular that
$$\text{Hom}(\calF, \tau_{\leq i+1}Rf_*\underline{\QQ}_{X} )  \cong \text{Hom}(\calF,\tau_{\leq i}Rf_*\underline{\QQ}_{X})\oplus \text{Hom}(\calF, R^{i+1}f_*\underline{\QQ}_X[-(i+1)])$$ for all $i$.  The proof of this claim is by induction on $i$.
 
First consider the $i=0$ case.  Here we are asking about the map $\text{Ext}^j(\calF,R^1f_*\underline{\QQ}_X)\to \text{Ext}^{j+1}(\calF, R^0f_*\underline{\QQ}_X)$.  By Lemma \ref{end}, the endomorphism $[N]$ acts on $\text{Ext}^j(\calF,R^1f_*\underline{\QQ}_X)$ with eigenvalue $N$ and on $\text{Ext}^{j+1}(\calF, R^0f_*\underline{\QQ}_X)$ trivially, and thus the map must be zero.

Now assume that the statement is true for all $k\leq i-1$.  This gives us a (non-canonical) isomorphism $\text{Ext}^{j+1}(\calF,\tau_{\leq k}Rf_*\underline{\QQ}_X) \cong \bigoplus_{i=0}^{k}\text{Ext}^{j+1}(\calF,R^if_*\underline{\QQ}_X[-i])$.  The map we are interested in can then be written as $$\text{Ext}^{j}(\calF,R^{k+1}f_*\underline{\QQ}_X[-(k+1)])\to \bigoplus_{i=0}^{k}\text{Ext}^{j+1}(\calF,R^if_*\underline{\QQ}_X[-i]),$$ which then decomposes into maps $\text{Ext}^{j}(\calF,R^{k+1}f_*\underline{\QQ}_X[-(k+1)])\to \text{Ext}^{j+1}(\calF,R^if_*\underline{\QQ}_X[-i])$ for $i=0,\dots,k$, each of which is zero since we have $[N]$ acting as $N^{k+1}$ on the domain and as $N^i$ on the range. The reader is encouraged to take a short break here to enjoy a beverage of their choice.

Moving on to the second part of the argument as described in the first paragraph of the proof, recall that the components of $H^1(P_\bullet,T^n_\CC)$ are indexed by the torsion elements of $H^1(P_\bullet,{\ZZ})$ by the universal coefficient theorem. This description, namely $H^1(P_\bullet,T^n_\CC) \cong \text{Hom}(H_1(P_\bullet,\ZZ),T^n_\CC)$ also tells us that there is an element whose non-torsion part is trivial in each component.

It only remains to show that $Rf_*\underline{\QQ}_X$ doesn't change as the torsor defining the gluing of the toric components of $X$ varies in a connected component.  Let $X_0$ be a broken toric variety over $P_\bullet$ for which the theorem holds, $X_1$ be a broken toric variety in the same component of $H^1(P_\bullet,T^n_\CC)$, and $\gamma:[0,1]$ be a path between them. This yields a family of broken toric varieties $\mathfrak{X}\xrightarrow{f} P_\bullet\times [0,1]$. Pushing forward the constant sheaf along the family of moment maps yields $\calF := Rf_*\underline{\QQ}_{\mathfrak{X}}$, a complex of sheaves on $P_\bullet\times [0,1]$. We will show that $\calF|_{P_\bullet\times\{t\}} \cong Rf_*\underline{\QQ}_{X_t}$ is independent of $t\in[0,1]$.

If $g$ is the contraction map of $[0,1]$ to $\{0\}$, then $g^{-1}g_*\calF$ is a sheaf on $P_\bullet$ and adjunction furnishes us with a morphism $g^{-1}g_*\calF \to \calF$. This is a quasi-isomorphism as can been seen by looking at the stalks of the cohomology sheaves of the complexes: we have $\left(\mathcal{H}^i(g^{-1}g_*\calF)\right)_{(x,t)} \cong \left(R^if_*\underline{\QQ}_{X_0}\right)_x$ and $\left(\mathcal{H}^i(\calF)\right)_{(x,t)}\cong \left(R^if_*\underline{\QQ}_{X_t}\right)_x$.  Both of these stalks are described entirely by $P_\bullet$, in particular independently of $t$, so they are isomorphic.  So $\calF|_{P_\bullet\times\{t\}} \cong Rf_*\underline{\QQ}_{X_t}$ is independent of $t$, and so the theorem holds for all broke toric varieties.
\end{proof}

The following corollary is the basis for all of our cohomology calculations.

\begin{corollary}\label{cohomcalc}
For any broken toric variety $X$
$$
H^k(X,\underline{\QQ}_X) = \bigoplus_{i+j=k}H^j(P_\bullet(X),R^if_*\underline{\QQ}_X).
$$
\end{corollary}

Having proven that the cohomology of a broken toric variety $X$ does not depend on how the components are glued together, we assume for the rest of the article that the torsor defining the gluing is trivial and in particular that $R^if_*\underline{\QQ}_X$ is a subsheaf of $\bigwedge\nolimits^i\underline{\QQ}_X^n$.

\section{Leray = Weight}\label{4}

To any map $f:X\to Y$ of algebraic varieties, one can assign the Leray filtration associated to $f$ to the cohomology groups of $X$. To do this, first apply the truncation functor $\tau_{\leq k}$ to the complex $Rf_*\underline{\QQ}_X$
\begin{align}
&\tau_{\leq k}\left(\cdots \xrightarrow{}(Rf_*\underline{\QQ}_X)^{k-1} \xrightarrow{d^{k-1}}(Rf_*\underline{\QQ}_X)^{k}\xrightarrow{d^{k}}\cdots \right) \nonumber
\\
&\qquad\qquad = \left( \cdots \xrightarrow{}(Rf_*\underline{\QQ}_X)^{k-1} \xrightarrow{d^{k-1}} \text{ker}d^{k}\xrightarrow{d^{k}} 0 \xrightarrow{}\cdots \right).\nonumber
\end{align}

The inclusion map $\tau_{\leq k}Rf_*\underline{\QQ}_X \to Rf_*\underline{\QQ}_X$ then induces a map
$$H^i(Y,\tau_{\leq k}Rf_*\underline{\QQ}_X)\xrightarrow{} H^i(Y,Rf_*\underline{\QQ}_X) \cong H^i(X,\underline{\QQ}_X),$$ the image of which is the $k$-th piece of the Leray filtration $L_k H^i(X,\underline{\QQ}_X).$

Theorem \ref{ff} then gives us an explicit formulation for the Leray filtration on the cohomology of a broken toric variety, namely
\begin{equation}\label{leray}
L_kH^i(X,\underline{\QQ}_X) = \bigoplus_{\substack{p+q=i \\ q\leq k}}H^p(P_\bullet,R^qf_*\underline{\QQ}_X).
\end{equation}

The Leray filtration can equivalently be described as the filtration arising from the degeneration of the Leray spectral sequence. In the same way, one obtains the weight filtration $W_kH^i(X,\underline{\QQ}_X)$ from Deligne's weight spectral sequence \cite{D71} (see \cite{GNPP88} for the singular case).

The main ingredient for proving the equivalence of these two filtrations is the following lemma. Note that given a polytope $P$, we denote by $X(P)$ the toric variety associated to it.

\begin{lemma}
Let $X$ be an $n$-dimensional broken toric variety with polytope complex $P_\bullet$ and set $\tilde{X} := \bigsqcup_{\overline{A}\in\text{Sk}_{n}(P_\bullet)}X(\overline{A})$.  In addition let $Y$ be the singular locus of $X$, i.e. if we set $$S = \text{Sk}_{n-1}(P_\bullet)\setminus \{\alpha\in\mathring{\text{Sk}}_{n-1}(P_\bullet):\alpha\in\partial A \text{ for only one } \overline{A}\in\text{Sk}_n(P_\bullet)\}$$ then $Y=X|_{f^{-1}(S)}$,
and let $\tilde{Y}$ be its normalization
${\tilde{Y} := \bigsqcup_{\overline{\alpha}\in S}X(\overline{\alpha})}.$

Then
\begin{enumerate}
\item There is a long exact sequence
\begin{equation}\label{4.1.1}
 \cdots\to H^j(X) \to H^j(Y) \oplus H^j(\tilde{X}) \to H^j(\tilde{Y}) \to\cdots
 \end{equation} which respects both the weight and Leray filtrations.
 \item  There is a long exact sequence
  \begin{align*}
 &\cdots\to H^j(P_\bullet, R^if_*\underline{\QQ}_{X}) \to H^j((P_\bullet)_Y, R^if_*\underline{\QQ}_{Y}) \oplus H^j((P_\bullet)_{\tilde{X}}, R^if_*\underline{\QQ}_{\tilde{X}})
 \\
 &\qquad\qquad \to H^j((P_\bullet)_{\tilde{Y}}, R^if_*\underline{\QQ}_{\tilde{Y}}) \to H^{j+1}(P_\bullet, R^if_*\underline{\QQ}_{X})\to\cdots
 \end{align*}
 which respects the weight filtration.
 \end{enumerate}
\end{lemma}
\begin{proof}
The first sequence is the Mayer-Vietoris sequence associated to the Cartesian square
\[\begin{tikzcd}
	{\tilde{Y}} && {\tilde{X}} \\
	Y && X
	\arrow[from=1-1, to=1-3]
	\arrow[from=2-1, to=2-3]
	\arrow[from=1-3, to=2-3]
	\arrow[from=1-1, to=2-1]
\end{tikzcd}\]
and it respects the weight filtration by construction (see \cite{GNPP88}). It respects the Leray filtration by virtue of the existence of the second sequence. That is to say, we can take the above square and look at the higher derived pushforwards to their polytope complexes and so get Mayer-Vietoris sequences
  \begin{align*}
 &\cdots\to H^j(P_\bullet, R^if_*\underline{\QQ}_{X}) \to H^j((P_\bullet)_Y, R^if_*\underline{\QQ}_{Y}) \oplus H^j((P_\bullet)_{\tilde{X}}, R^if_*\underline{\QQ}_{\tilde{X}})
 \\
 &\qquad\qquad \to H^j((P_\bullet)_{\tilde{Y}}, R^if_*\underline{\QQ}_{\tilde{Y}}) \to H^{j+1}(P_\bullet, R^if_*\underline{\QQ}_{X})\to\cdots
 \end{align*}
 for each $i$. Taking direct sums of these sequences (in view of the description \eqref{leray} of the Leray filtration) shows that the sequence of \eqref{4.1.1} respects the Leray filtration. 
 
\end{proof}

\begin{theorem}\label{L=W}
For $X$ a broken toric variety, $$W_{2k}H^i(X,\underline{\QQ}_X) = W_{2k+1}H^i(X,\underline{\QQ}_X) = L_{k}H^i(X,\underline{\QQ}_X).$$
\end{theorem}

\begin{proof}
For smooth toric varieties it is well-known that their odd degree cohomology vanishes and that they have pure cohomology, which in this context means that $W_kH^{2i}(X,\underline{\QQ}_X) = 0$ for all $k \leq 2i-1$ and that $W_{2i}H^{2i}(X,\underline{\QQ}_X) = H^{2i}(X,\underline{\QQ}_X)$.  On the Leray side, Lemma \eqref{j<ivanishing} and Corollary \eqref{j>ivanishing} together show that $H^p(P_\bullet,R^qf_*\underline{\QQ}_X)$ is only nonzero for $p=q$, so that $L_kH^{2i}(X,\underline{\QQ}_X) = 0$ for $k\leq i-1$ and that $L_iH^{2i}(X,\underline{\QQ}_X) = H^{2i}(X,\underline{\QQ}_X)$.

To extend this result to all broken toric varieties, we first notice that the statement of the theorem is equivalent to asking that the weight filtration of the sheaf $R^if_*\underline{\QQ}_X$ on $P_\bullet$ is pure of weight $2k$, or in other words that 
\begin{equation}\label{purity}
W_{2k}H^j(P_\bullet, R^if_*\underline{\QQ}_X) = W_{2k+1}H^j(P_\bullet, R^if_*\underline{\QQ}_X) = H^j(P_\bullet, R^if_*\underline{\QQ}_X)
\end{equation}
for all $k\geq i$. This is because, again in view of Theorem \ref{ff}, we have
$$W_{2k}H^p(X,\underline{\QQ}_X) = W_{2k}\bigoplus_{i+j=p}H^j(P_\bullet, R^if_*\underline{\QQ}_X) = \bigoplus_{i+j=p}W_{2k}H^j(P_\bullet, R^if_*\underline{\QQ}_X)$$
which is equal to 
$$
 \bigoplus_{\substack{i+j=p \\ i\leq k}}H^j(P_\bullet,R^if_*\underline{\QQ}_X) = L_kH^p(X,\underline{\QQ}_X)$$
 if and only if Equation \eqref{purity} is true (and similary for $W_{2k+1}$).

The result now follows by induction. Starting with $n=1$, we see that the result is trivially true for $Y$ and $\tilde{Y}$, and true for $\tilde{X}$ since it is a disjoint union of toric varieties. To show that \eqref{purity} is true for $X$, consider the inclusion map of complexes

\[\begin{tikzcd}[column sep=tiny]
	\cdots & { W_{2k}H^j(R^if_*\underline{\mathbb{Q}}_{X})} & { W_{2k}H^j(R^if_*\underline{\mathbb{Q}}_{Y}) \oplus W_{2k}H^j(R^if_*\underline{\mathbb{Q}}_{\tilde{X}})} & { W_{2k}H^j(R^if_*\underline{\mathbb{Q}}_{\tilde{Y}})} & \cdots \\
	\cdots & {H^j(R^if_*\underline{\mathbb{Q}}_{X})} & { H^j(R^if_*\underline{\mathbb{Q}}_{Y}) \oplus H^j(R^if_*\underline{\mathbb{Q}}_{\tilde{X}})} & { H^j(R^if_*\underline{\mathbb{Q}}_{\tilde{Y}})} & \cdots
	\arrow[from=1-2, to=1-3]
	\arrow[from=1-3, to=1-4]
	\arrow[from=1-4, to=1-5]
	\arrow[from=1-1, to=1-2]
	\arrow[from=2-3, to=2-4]
	\arrow[from=2-4, to=2-5]
	\arrow[from=2-1, to=2-2]
	\arrow[from=2-2, to=2-3]
	\arrow[hook, from=1-2, to=2-2]
	\arrow["\sim", from=1-3, to=2-3]
	\arrow["\sim", from=1-4, to=2-4]
\end{tikzcd}\]

We see that the inclusion $W_{2k}H^j(R^if_*\underline{\mathbb{Q}}_{X})\hookrightarrow H^j(R^if_*\underline{\mathbb{Q}}_{X})$ is in fact an isomorphism by the $5$-lemma and so the proof is complete.

\end{proof}

\section{Cell-Compatible Sheaves on Polytopes}\label{5}

Here, we will define a sequence of cell-compatible sheaves which are then used reproduce the well-known formula for the Betti numbers of a toric variety, and some related facts about certain cell-compatible sheaves on non-simple polytopes. 

Take any polytope $P$ of dimension $n$ and define a particular sequence of subsets $A_m$ of $\mathring{\text{Sk}}_{n-1}(P)$ by letting $a_1,\dots,a_{|\mathring{\text{Sk}}_{n-1}(P)|}$ be an enumeration of $\mathring{\text{Sk}}_{n-1}(P)$ such that $A_m := \bigcup_{i=1}^m \overline{a_m}$ is contractible for all $m<|\mathring{\text{Sk}}_{n-1}(P)|$.  Define a subsheaf $\calS^i_P(A_m)$ of $\bigwedge\nolimits^i\underline{\QQ}_P^n$ by the following restriction of the stalks:
$$\left(\calS^i_P(A_m)\right)_x = \bigwedge\nolimits^i\left(\bigcap_{\substack{ \overline{a}_j\ni x \\ 0\leq j\leq m }}H_{a_j}\right),$$
where $V$ is the ambient $n$-dimensional vector space and $H_{a_j} \subset V$ is the central hyperplane associated to the $n-1$ cell $a_j$. 

The following proposition is a bit technical and its use is essentially to prove Corollary \ref{j>ivanishing}.

\begin{proposition}\label{t}
Let $P$ be a polytope of dimension $n$ and $\{A_m\}$ a sequence of subsets of $\mathring{\text{Sk}}_{n-1}(P)$ as defined above. Then:
\begin{enumerate}
\item  $\calS^i_{P}(\varnothing) \cong \bigwedge\nolimits^i\underline{\QQ}_{P}^n$ and if $P$ is the polytope of a projective toric variety $X$ then $\calS^i_{P}(\mathring{\text{Sk}}_{n-1}) \cong R^if_*\underline{\QQ}_X$.
\item If $P$ is a simple polytope there is a short exact sequence of sheaves 
$$ 0\to\calS^i_{P}(A_m) \to \calS^i_{P}(A_{m-1}) \to \calS^{i-1}_{a_m}(A_{m-1}\cap \{a_m\})\to 0.$$
\item $H^j(\calS^i_P(A_m)) = 0$ for all $i<j$. If $P$ is a simple polytope then in addition $H^j(\calS^i_P(A_m)) = 0$ for all $j>i$.
\end{enumerate}
\end{proposition}

\begin{proof}
The first part of (1) follows directly from the definition, as does the second since we know that $R^if_*\underline{\QQ}_X$ is the subsheaf of $\bigwedge\nolimits^i\underline{\QQ}_P^{n}$ with stalks $$(R^if_*\underline{\QQ}_X)_x = \bigwedge\nolimits^i\left( \bigcap_{\substack{a\in\mathring{\text{Sk}}_{n-1}(P) \\ \overline{a}\ni x}}H_a\right)$$
(this is simply another way of writing the description of the stalks that appears in the proof of Proposition \ref{cc}).

For (2) consider the surjective morphism of sheaves 
$$q:\calS^i_{P}(A_{m-1}) \to \calS^{i-1}_{a_m}(A_{m-1}\cap\{a_m\})$$
defined on stalks in the following way: First, if $x\not\in\overline{a_m}$ then $q_x$ is the zero map and if $x\in\overline{a_m}$ then
$$q_x:\bigwedge\nolimits^i\left(\bigcap_{\overline{a_k}\ni x}H_{a_k}\right)\to \bigwedge\nolimits^i\left(\bigcap_{\overline{a_k}\ni x}H_{a_k}\cap H_{a_m}\right)$$
where $\bigcap_{\overline{a_k}\ni x}H_{a_k}$ is an $l$-dimensional vector space with a chosen basis $\{b_1,\dots,b_l\}$ and so $\bigcap_{\overline{a_k}\ni x}H_{a_k}\cap H_{a_m}$ is $(l-1)$-dimensional with basis $\{b_1,\dots,b_{l-1}\}$. Then $q_x$ is defined as taking $b_{r_1}\wedge\dots\wedge b_{r_i}$ to $0$ if $b_l$ is not among the $b_r$, and to $b_{r_1}\wedge\dots\wedge \hat{b_{l}}\wedge\dots\wedge b_{r_i}$ otherwise.

The kernel of $q$ is a subsheaf of $\bigwedge\nolimits^i\underline{\QQ}_{P}^{n}$ (since it is a subsheaf of $\calS^i_{P}(A_{m-1})$) whose stalks match those of $\calS^i_{P}(A_m)$, hence it is $\calS^i_{P}(A_m)$ itself.

(3) follows in the same way as the proof of Lemma \ref{j<ivanishing}. For the second statement we proceed by induction on the dimension $n$ of $P$. For the $1$-dimensional simple polytope $P$, we can see
$$\calS^0_P(A)\cong R^0f_*\underline{\QQ}_{X}\text{ for all } A\subseteq \mathring{\text{Sk}}_0(P)$$
and
$$\calS^1_P(\varnothing) \cong \underline{\QQ}_{P},$$
$$\calS^1_P({\{a_i\}})\cong j_!\underline{\QQ}_{P\setminus\{a_i\}},$$
$$\calS^1_P({\{a_1,a_2\}})\cong j_!\underline{\QQ}_{P\setminus\{a_1, a_2\}} \cong R^1f_*\underline{\QQ}_{X}$$
These are fairly innocent cell-compatible sheaves with the property that $H^j(\calS^i_P(A)) = 0$ for all $j>i$.  Assuming that this holds for a dimension $n$ polytope, consider the long exact sequence in cohomology arising from the short exact sequence of (2) with $a_m = P$ and $P_\bullet$ equal to some $n+1$-dimensional polytope with $P$ as a face. Note that it is here where we use simplicity of the polytope; the statement of (2) does not hold for $P$ non-simple. This gives the desired result for $n+1$, since we have in addition that for any $n$, $H^j(P,\calS_P^i(\varnothing)) = H^j(P,\bigwedge\nolimits^i\underline{\QQ}^n_P) =0$ for all $j>i$.

\end{proof}

\begin{corollary}\label{j>ivanishing}
Let $X$ be a smooth projective toric variety with polytope $P$. Then $H^j(P,R^if_*\underline{\QQ}_X) = 0$ for all $j > i$.
\end{corollary}

\begin{proof}
Combine (1) and (3) of Proposition \ref{t}.
\end{proof}

\begin{corollary}[\cite{D78}]\label{toricbetti} Let $X$ be a smooth projective toric variety of dimension $n$ with polytope $P$. Then the odd-dimensional cohomology of $X$ vanishes and 
\begin{equation*}
h^{2i}(X,\underline{\QQ}_X) = h^i(P,R^if_*\underline{\QQ}_X) = \sum_{j=i}^n(-1)^{i-j}\binom{j}{i}|\mathring{\text{Sk}}_{j}(P)|.
\end{equation*}
\end{corollary}

\begin{proof}
Assume that $i>0$, since the $i=0$ case is trivial. Denote $\calS^i_P(\mathring{\text{Sk}}_{n-1}(P_\bullet))$ by $\calF$ and consider the long exact sequence associated to \eqref{restrictiontoskeletongeneral}, which  tells us that 
\begin{equation}\label{easyiso}
H^l(P,(\kappa_k)_!(\kappa_k)^{-1}\calF) = H^l(P,(\kappa_{k-1})_!(\kappa_{k-1})^{-1}\calF)
\end{equation} for all $l\neq k,k-1$, since $H^l(P,(\iota_k)_!\calF|_{\mathring{\text{Sk}}_{k}(P)}) = H^l(P,(\iota_{l+1})_!\bigwedge\nolimits^i\underline{\QQ}^{n}_{\mathring{\text{Sk}}_{k}(P)} )$ is only nonzero for $l=k$. Since $H^l(P,(\kappa_0)_!(\kappa_0)^{-1}\calF)=0$ for all $l$, we find
$$H^l(P,(\kappa_k)_!(\kappa_k)^{-1}\calF) = 0$$ for all $l>k$. On the other hand, setting $k=n$ in \eqref{easyiso} yields 
\begin{equation*}
H^l(P,\calF) = H^l(P,(\kappa_{n-1})_!(\kappa_{n-1})^{-1}\calF)
\end{equation*} for all $l\neq n,n-1$, and so in particular by Corollary \ref{j>ivanishing} we have $$H^l(P,(\kappa_{n-1})_!(\kappa_{n-1})^{-1}\calF)=0$$ for all $l\neq n,n-1,i$. We can again use \eqref{easyiso} to say that 
\begin{equation*}
H^l(P,(\kappa_k)_!(\overline{\iota }_k)^{-1}\calF) =0
\end{equation*}
for all $i\neq l<k$ (as well as for all $k<i$, as in the proof of Lemma \ref{openclosed}).  These vanishings together imply that
\begin{equation}\label{blub}
h^l(P,(\kappa_{l})_!(\kappa_{l})^{-1}\calF) = h^l(P,(\iota_l)_!\bigwedge\nolimits^i\underline{\QQ}^{n}_{\mathring{\text{Sk}}_{k}(P)} ) -h^{l-1}(P,(\kappa_{l-1})_!(\kappa_{l-1})^{-1}\calF)
\end{equation}
 for $l\leq i$, and
\begin{equation}\label{blab}
h^l(P,(\kappa_{l})_!(\kappa_{l})^{-1}\calF) = h^{l+1}(P,(\iota_{l+1})_!\bigwedge\nolimits^i\underline{\QQ}^{i+1}_{\mathring{\text{Sk}}_{k}(P)} ) -h^{l+1}(P,(\kappa_{l+1})_!(\kappa_{l+1})^{-1}\calF)
\end{equation}
for $l>i$.

Of course, what we are actually interested in $H^i(P,\calF)$, which is isomorphic to $H^i(P,(\kappa_{i+1})_!(\kappa_{i+1})^{-1}\calF)$ by \eqref{easyiso}. The long exact sequence associated to \eqref{restrictiontoskeletongeneral} for $k=i+1$ tells us that 
\begin{align}\label{center}
h^i(P,\calF) &= h^i(P,(\kappa_{i})_!(\kappa_{i})^{-1}\calF) -h^{i+1}(P,(\iota_{i+1})_!\bigwedge\nolimits^i\underline{\QQ}^{i+1}_{\mathring{\text{Sk}}_{i+1}(P_\bullet)})+\\
&\qquad+ h^{i+1}(P,(\kappa_{i+1})_!(\kappa_{i+1})^{-1}\calF).\nonumber
\end{align}

Repeatedly applying Equations \eqref{blub} and \eqref{blab} calculates
\begin{align*}
h^i(P,(\kappa_{i})_!(\kappa_{i})^{-1}\calF) &= \sum_{j=0}^i(-1)^{i-j}h^{j}(P,(\iota_{j})_!\bigwedge\nolimits^i\underline{\QQ}^j_{\mathring{\text{Sk}}_{j}(P)})\\
&=\sum_{j=0}^i(-1)^{i-j}\binom{j}{i}|\mathring{\text{Sk}}_{j}(P)|\\
&=|\mathring{\text{Sk}}_{i}(P)|
\end{align*}
and
\begin{align*}
h^{i+1}(P,(\kappa_{i+1})_!(\kappa_{i+1})^{-1}\calF) &= \sum_{j=i+2}^n(-1)^{i-j}h^{j}(P,(\iota_{j})_!\bigwedge\nolimits^i\underline{\QQ}^j_{\mathring{\text{Sk}}_{j}(P)})\\
&= \sum_{j=i+2}^n(-1)^{i-j}\binom{j}{i}|\mathring{\text{Sk}}_{j}(P)|
\end{align*}
so that \eqref{center} simplifies to 
\begin{align*}
h^i(P,\calF) &=
|\mathring{\text{Sk}}_{i}(P)| - \binom{i+1}{i}|\mathring{\text{Sk}}_{i+1}(P)|+\\
&\qquad+ \sum_{j=i+2}^n(-1)^{i-j}\binom{j}{i}|\mathring{\text{Sk}}_{j}(P)| \\
&= \sum_{j=i}^n(-1)^{i-j}\binom{j}{i}|\mathring{\text{Sk}}_{j}(P)|
\end{align*}

This is an expression for the $2i$-th Betti number of $X$ thanks to Theorem \ref{ff} and the vanishings of Lemma \ref{j<ivanishing} and Corollary \ref{j>ivanishing}.
\end{proof}

\section{Broken Toric Varieties with Skeletal Polytope Complexes}\label{6}

In this section we describe the cohomology of a broad class of broken toric varieties, namely those $X$ whose polytope complexes comprise the $n$-skeleton of a higher dimensional polytope $P'$. The idea is that the sheaves $R^if_*\underline{\QQ}_X$ which come into play here are subsheaves of the cell-compatible sheaves $\mathcal{T}^i_{P'}(\mathring{\text{Sk}}_{\text{dim}(P')-1}(P'))$ (which we understand by Section \ref{5}), and we can use this relationship to extract information about the former.

\begin{definition}
We say that a polytope complex $P_\bullet$ of dimension $n$ is \emph{skeletal} if there exists a polytope $P'$ such that $P_\bullet = \text{Sk}_n(P')$.
\end{definition}

To get a feel for this definition, note that the polytope complex in Example \ref{ex1} is skeletal while the one in Example \ref{ex2} is not. 

\begin{proposition}
If $X$ is a broken toric variety of dimension $n$ with skeletal polytope complex $P_\bullet$ and $P'$ is a polytope of dimension $n'$ such that $P_\bullet = \text{Sk}_n(P')$, then
\begin{equation*}
h^j(P_\bullet,R^if_*\underline{\QQ}_{P_\bullet}) = 
\begin{cases}
h^j(P',R^if_*\underline{\QQ}_{X(P')}) , &i \leq j < n\\
\sum_{l=n}^{n'}(-1)^{n+l}h^l(P',R^if_*\underline{\QQ}_{X(P')})   &\\
 \qquad + \sum_{l=n+1}^{n'}(-1)^{n+l+1}\binom{l}{i}|\mathring{\text{Sk}}_l(P')|, & i\leq j=n \\
0, & \text{otherwise}
\end{cases}
\end{equation*}
In particular, if $P'$ is polytope, then
\begin{equation*}
h^j(P_\bullet,R^if_*\underline{\QQ}_{P_\bullet}) = 
\begin{cases}
h^i(P',R^if_*\underline{\QQ}_{X(P')}) , &i = j < n\\
\sum_{l=n+1}^{n'}(-1)^{n+l+1}\binom{l}{i}|\mathring{\text{Sk}}_l(P')|  , & i<j=n \\
h^n(P', R^nf_*\underline{\QQ}_{X(P')}) + \sum_{l=n+1}^{n'}(-1)^{n+l+1}\binom{l}{i}|\mathring{\text{Sk}}_l(P')|   , & i=j=n \\
0, & \text{otherwise}
\end{cases}
\end{equation*}
\end{proposition}

\begin{proof}
This result follows from inductively applying Lemma \ref{openclosed} first to $\calF=R^if_*\underline{\QQ}_{Q}$ for $k=\text{dim}(Q)$, then to $\calF = R^if_*\underline{\QQ}_{\text{Sk}_{\text{dim}(Q)-1}(Q)}$ for $k=\text{dim}(Q)-1$, etc. In the first step we find the desired formula by the vanishings given by Lemma \ref{j<ivanishing}. For $P'$ polytope, we apply  Lemma \ref{j>ivanishing} to get the second part of the proposition.
\end{proof}

The second part of this proposition allows for an interested party to write out an explicit formula for the Betti numbers of any broken toric variety whose polytope complex is the $n$-skeleton of a higher dimensional simple polytope. We will content ourselves with writing this out for the case $\text{dim}(P') = n+1$.

\begin{corollary}
If $X$ is a broken toric variety of dimension $n$ with skeletal polytope complex $P_\bullet = \text{Sk}_n(P')$ for $P'$ a simple polytope of dimension $n+1$, then
\begin{equation*}
h^j(P_\bullet,R^if_*\underline{\QQ}_{P_\bullet}) = 
\begin{cases}
h^j(P',R^if_*\underline{\QQ}_{X(P')}) , &i = j < n\\
\binom{n+1}{i}, & i < j = n\\
{|\mathring{\text{Sk}}_n(P_\bullet)|}, & i=j=n \\
0, & \text{otherwise}
\end{cases}
\end{equation*}
and
\begin{equation*}
h^j(X,\underline{\QQ}_X) = 
\begin{cases}
(-1)^{n+1-2i}\binom{n+1}{i}+\sum_{k=i}^n(-1)^{i-k}\binom{k}{i}|\mathring{\text{Sk}}_{k}(P)| , &j=2i< n\\
0, & j=2i+1<n\\
\binom{n+1}{2i+1-n}, & n \leq j=2i+1\leq 2n \\
\binom{n+1}{2i-n}+(-1)^{n+1-2i}\binom{n+1}{i}+\sum_{k=i}^n(-1)^{i-k}\binom{k}{i}|\mathring{\text{Sk}}_{k}(P)|, & n\leq j=2i\leq 2n
\end{cases}
\end{equation*}
\end{corollary}

\appendix 

\section{Broken Toric Varieties as Quotients of Polyhedral Products}\label{A}

In this appendix we will describe how broken toric varieties appear in the study of polyhedral products, further cementing them as objects with widespread appeal.

Given a family of based CW pairs $(\underline{X},\underline{A}) := (X_i,A_i)_{i=1}^m$ and a simplicial complex $K$ on $m$ vertices, the polyhedral product is a tool developed in the field of toric topology to produce a subspace $\calZ(K,(\underline{X},\underline{A}))$ of the Cartesian product $X_1\times\dots\times X_m$. To be precise, for $\mathbf{SCpx}$ the category of simplical complexes and $\calC [m]$ the category of $m$-tuples of based CW pairs, it is a functor 
$$\calZ(-,-) : \mathbf{SCpx} \times \calC [m] \to \mathbf{Top}$$
satisfying (following \cite{BBC20})
$$\calZ(K,(\underline{X},\underline{A})) \subseteq X_1\times\dots\times X_m$$
and which is the colimit of a diagram $D$ in the category $CW_*$ of pointed CW pairs, defined for $\sigma\in K$ by
$$W_i = \begin{cases}
X_i, &i\in\sigma \\
A_i, &i\not\in\sigma
\end{cases}$$
and 
$$D(\sigma) = W_1\times \dots \times W_m.$$

If we fix $(\underline{X},\underline{A})$ with $A_i$ the basepoint of $X_i$ for all $i$, one can think of the polyhedral product for different $K$ as interpolating between $X_1\vee\dots\vee X_m$ (when $K$ is $m$ discrete points) and $X_1\times\dots\times X_m$ (when $K$ is the full $(m-1)$-simplex). 

Another relevant example is the case where $(\underline{X},\underline{A}) = ((\PP^2, \PP^1),(\PP^2,\PP^1))$ and $K$ is the simplicial complex consisting of $2$ distinct points. Here we find that
$$\calZ(\{\{1\},\{2\}\},((\PP^2, \PP^1),(\PP^2,\PP^1))) \cong \PP^2\times \PP^1 \cup_{\PP^1\times\PP^1} \PP^1\times \PP^2.$$ The diagonal torus $\Delta = \CC^*$ in $\PP^1 \times \PP^1$ acts in a natural way with Chow quotient $\PP^1$ and further on $\PP^2\times \PP^1$ (and $\PP^1\times\PP^2$) with Chow quotient $\PP^2$.  Thus, taking the quotient of $\calZ(\{\{1\},\{2\}\},((\PP^2, \PP^1),(\PP^2,\PP^1)))$ by this diagonal subgroup yields the broken toric variety of Example \ref{ex2}.

The above example generalizes. For any two (potentially broken) toric varieties $X$ and $X'$ which we wish to glue along a common (potentially broken) toric subvariety $Y\subset X, X'$, take the polyhedral product 
$$\calZ:=\calZ(\{\{1\},\{2\}\}, ((X,Y),(X', Y))) = X\times Y \cup_{Y\times Y} Y\times X$$ and quotient out by the diagonal torus $\Delta\subset Y\times Y$. This quotient is best viewed from the point of view of the corresponding fans (\emph{cf.} \cite{KSZ91}); the quotient of a toric variety $X$ with fan $\calF_X\subset T^{\vee}_\RR$ by a subgroup $H$ of its torus can be described as the toric variety associated to the fan one gets by projecting $\calF_X$ onto the rational subspace of $T^{\vee}_\RR$ defined by $H$. So, the quotient of $Y\times Y$ by $\Delta$ is the toric variety associated to the fan $\calF_Y\times \calF_Y$ projected to the diagonal hyperplane, which is just $\calF_Y$ itself.  In a similar way, the quotient of $X\times Y$ by $\Delta$ is the toric variety associated to the fan $\calF_X \times \calF_Y$ projected to the diagonal hyperplane, which is $\calF_X$.  All together, the quotient of $\calZ$ by $\Delta$ is the desired broken toric variety. 

Iterating the above construction shows that any broken toric variety is the quotient of a polyhedral product.

\bibliographystyle{acm} 
\bibliography{CCSheavesBibliography}

\end{document}